\documentclass[11pt]{article}
\usepackage{amsthm, amsmath, amssymb, amsfonts, url, booktabs, tikz, setspace, fancyhdr}
\usepackage[hidelinks]{hyperref}
\usepackage[margin = 1in]{geometry}
\usepackage{hyperref, enumerate}
\usepackage[shortlabels]{enumitem}
\usepackage[babel]{microtype}
\usepackage[english]{babel}
\usetikzlibrary{arrows.meta}

\usepackage{enumitem}
\newlist{steps}{enumerate}{1}
\setlist[steps, 1]{label = Step \arabic*:}

\newtheorem{theorem}{Theorem}[section]
\newtheorem{proposition}[theorem]{Proposition}
\newtheorem{lemma}[theorem]{Lemma}
\newtheorem{corollary}[theorem]{Corollary}
\newtheorem{conjecture}[theorem]{Conjecture}

\newtheorem{question}[theorem]{Question}

\theoremstyle{definition}

\theoremstyle{remark}
\newtheorem*{remark}{Remark}

\newcommand{\norm}[1]{\left\lVert#1\right\rVert}

\newcommand{\overbar}[1]{\mkern 1.5mu\overline{\mkern-1.5mu#1\mkern-1.5mu}\mkern 1.5mu}

\newcommand{\x}{\times}

\newcommand{\W}{\mathcal{W}}

\newcommand{\Sy}{\mathrm{Sym}}

\usetikzlibrary{trees}
\tikzstyle{p}+=[fill=black, circle, minimum width = 1pt, inner sep =
1pt]
\tikzstyle{w}+=[fill=white, draw, circle, minimum width = 1pt, inner sep =
1.5pt]

\begin{document}

\title{Convex graphon parameters and graph norms}

\author{Joonkyung Lee  \thanks{E-mail: {\tt joonkyung.lee@uni-hamburg.de}. Fachbereich Mathematik, Universit\"at Hamburg, Germany.
 Research supported by ERC Consolidator Grant PEPCo 724903.}
 \and Bjarne Sch\"{u}lke \thanks{E-mail: {\tt  bjarne.schuelke@uni-hamburg.de}. Fachbereich Mathematik, Universit\"at Hamburg, Germany.
 Research supported by G.I.F. Grant Agreements No. I-1358-304.6/2016.}}

\date{}

\maketitle

\begin{abstract}
Sidorenko's conjecture states that the number of copies of a bipartite graph $H$ in a graph~$G$ is asymptotically minimised when $G$ is a quasirandom graph. A notorious example where this conjecture remains open is when $H=K_{5,5}\setminus C_{10}$. It was even unknown whether this graph possesses the strictly stronger, weakly norming property.

We take a step towards understanding the graph $K_{5,5}\setminus C_{10}$ by proving that it is not weakly norming. More generally, we show that `twisted' blow-ups of cycles, which include $K_{5,5}\setminus C_{10}$ and $C_6\square K_2$, are not weakly norming. This answers two questions of Hatami.
The method relies on the analysis of Hessian matrices defined by graph homomorphisms, by using the equivalence between the (weakly) norming property and convexity of graph homomorphism densities.
We also prove that $K_{t,t}$ minus a perfect matching, proven to be weakly norming by Lov\'asz, is not norming for every~$t>3$. 
\end{abstract}

\section{Introduction}
In extremal combinatorics, quantifying quasirandomness by using a suitable norm has  been an extremely useful strategy. For instance, the main idea in the proof of the celebrated Szemer\'edi regularity lemma is to use an $L^2$-norm increment, the Gowers norms play a central role in additive combinatorics, and the cut-norm is the key concept in the theory of dense graph limits~\cite{LSz06}.

It is a natural question to ask what norms can be defined on the space of two-variable real symmetric functions on~$[0,1]^2$, which appear to be the limit objects of sequences of (weighted) large graphs. To formalise, a \emph{graphon} (resp.~\emph{signed graphon}) $W$ is a two-variable symmetric measurable function from~$[0,1]^2$ to~$[0,1]$ (resp.~$[-1,1])$. 
We consider the vector space $\W$ of two-variable symmetric bounded measurable functions on $[0,1]^2$, which contains the set of (signed) graphons as a convex subset.
Given a graph $H$ and $W\in\W$, the \emph{homomorphism density} of $H$ is defined by the functional
\begin{align*}
    t_H(W) = \int \prod_{ij\in E(H)} W(x_i,x_j) d\mu^{v(H)},
\end{align*}
where $\mu$ is the Lebesgue measure on $[0,1]$.

Let $\norm{W}_H:=|t_H(W)|^{1/e(H)}$ and let $\norm{W}_{r(H)}:=t_H(|W|)^{1/e(H)}$.
We then say that a graph $H$ is \emph{(semi-)norming} if $\norm{\cdot}_H$ defines a (semi-)norm on $\W$, and \emph{weakly norming} if $\norm{\cdot}_{r(H)}$ is a norm on~$\W$.
With this notation, we now state the following central question in the area, asked by Lov\'asz~\cite{L08} and Hatami~\cite{H10}:
\begin{question}[\cite{L08}, Problem 24]\label{q:lovasz}
What graphs $H$ are (weakly) norming?
\end{question}

A moment's thought will prove the fact that a weakly norming graph $H$ must be biparitite and that, as the name suggests, every (semi-)norming graph is weakly norming.
The particular example~$\norm{\cdot}_{C_{2k}}$, where $C_{2k}$ is the even cycle of length $2k$, is already interesting, as it corresponds to the Schatten--von Neumann norms in operator theory.

Perhaps one of the most important applications of weakly norming graphs is to Sidorenko's conjecture, a major open problem in extremal graph theory also proposed by Erd\H{o}s and Simonovits~\cite{ESi83} in a slightly different form.
\begin{conjecture}[Sidorenko's conjecture~\cite{Sid92}]
Let $H$ be a bipartite graph and let $W$ be a graphon. Then 
\begin{align}\label{eq:Sido}
    t_H(W)\geq t_{K_2}(W)^{e(H)}.
\end{align}
\end{conjecture}
If a graph $H$ satisfies~\eqref{eq:Sido} for every graphon $W$, then we say that $H$ is \emph{Sidorenko}. 
Szegedy observed\footnote{It appeared in~\cite{H10}.} that every weakly norming graph is Sidorenko. 
Moreover, Conlon and the first author~\cite{CL16} proved that weakly norming graphs can be used as `building blocks' to construct a Sidorenko graph. On the other hand, there are Sidorenko graphs that are verified to be not weakly norming. For instance, a bipartite graph that has a vertex adjacent to all the vertices on the other side, proven to be Sidorenko by Conlon, Fox, and Sudakov~\cite{CFS10}, is not weakly norming unless it is a complete bipartite graph. 
Moreover, Kr\'al', Martins, Pach, and Wrochna~\cite{KMPW19} recently proved that  there exists an edge-transitive Sidorenko graph that is not weakly norming.

Although the weakly norming property is strictly stronger than being Sidorenko, 
partial answers to Question~\ref{q:lovasz} have also made significant progress towards Sidorenko's conjecture. Hatami~\cite{H10}, who firstly studied Question~\ref{q:lovasz}, showed that even cycles $C_{2k}$ are norming, and complete bipartite graphs $K_{m,n}$ and hypercubes $Q_{d}$ are weakly norming. Lov\'asz~\cite{L12} later proved that $K_{n,n}$ minus a perfect matching is weakly norming. Before their work, $Q_d$ and  $K_{n,n}$ minus a perfect matching were unknown to be Sidorenko. Recently, Conlon and the first author~\cite{CL16} obtained a much larger class of weakly norming graphs, which also added many new examples to the class of Sidorenko graphs that played a crucial r\^{o}le in their subsequent work~\cite{CL18}.

Despite a fair amount of recent progress~\cite{CFS10, CKLL15,CL16,CL18,H10,KLL14,LSz12,Sz15}, Sidorenko's conjecture remains open. In particular, none of the partial results succeeded in determining whether the notorious \emph{M\"obius ladder} graph  $K_{5,5}\setminus C_{10}$, suggested by Sidorenko~\cite{Sid93,Sid92}, is Sidorenko or not, although Conlon and the first author~\cite{CL18} proved that its `square' is Sidorenko. 
We make some progress in understanding this mysterious graph, by proving that it is not weakly norming. 
\begin{theorem}\label{thm:mobius}
The M\"obius ladder graph $K_{5,5}\setminus C_{10}$ is not weakly norming.
\end{theorem}

For a graph $H$, let $H^{\bowtie}$ be the graph obtained by blowing up every vertex $v$ of $H$ by an edge~$v_1v_2$ and putting two edges $u_2v_1$ and $u_1v_2$ between each pair of blown-up edges $u_1u_2$ and $v_1v_2$  
whenever~$uv\in E(H)$. The resulting graph $H^{\bowtie}$ is always a bipartite graph whose bipartite adjacency matrix is the (symmetric) adjacency matrix of $H$ plus the identity.
This blow-up was considered by Kim, Lee, and the first author~\cite{KLL14}. They observed (see Figure~\ref{fig:cycles}) that $C_5^{\bowtie}$ is isomorphic to the M\"obius ladder and, if $H$ is bipartite, $H^{\bowtie}$ is isomorphic to $H\square K_2$, where~$\square$ denotes the \emph{Cartesian product} of graphs. In particular, $C_4^{\bowtie}$ is the 3-cube graph, proven to be weakly norming by Hatami. We prove a more general result that implies Theorem~\ref{thm:mobius}.
\begin{theorem}\label{thm:bowtie}
For every $k>4$, $C_k^{\bowtie}$ is not weakly norming.
\end{theorem}

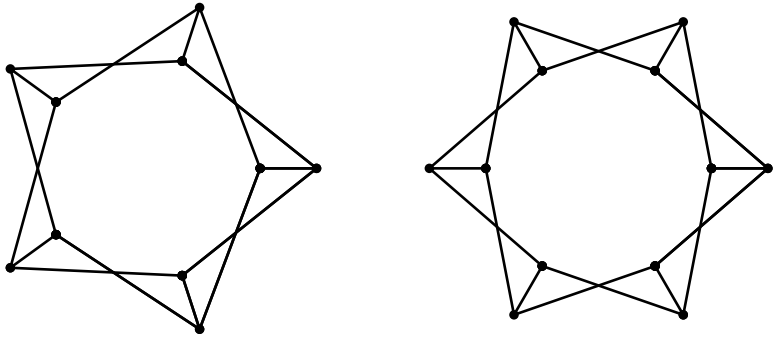
\begin{figure}
    \centering
    \begin{tikzpicture}[scale=1.5]
	
	\begin{scope}
	\foreach \x [count=\p] in {0,...,7} {
		\node[shape=circle, inner sep=0.0mm, minimum size=0.0mm, opacity=0] (\p) at (-\x*72:1.5) {};
		\fill  (\p) circle (1.2pt);};

	\foreach \x [count=\p] in {0,...,7} {
		\node[shape=circle, inner sep=0.0mm, minimum size=0.0mm, opacity=0] (\p2) at (-\x*72:1) {};
		\fill  (\p2) circle (1.2pt);};
	\foreach \x [count=\p] in {0,...,7} {
		\node[shape=circle, inner sep=0.0mm, minimum size=0.0mm, opacity=0] (\p3) at (-\x*72+72:1) {};
		\fill  (\p3) circle (1.2pt);};
	\foreach \x [count=\p] in {0,...,7} {
		\node[shape=circle, inner sep=0.0mm, minimum size=0.0mm, opacity=0] (\p4) at (-\x*72-72:1) {};
		\fill  (\p4) circle (1.2pt);};

	\foreach \x [count=\p] in {1,...,7} {
		\draw[black!100!white, line width=1pt] (\p) -- (\p2);
		\draw[black!100!white, line width=1pt] (\p) -- (\p3);
		\draw[black!100!white, line width=1pt] (\p) -- (\p4);
		};
	
	\end{scope}
	\begin{scope}[xshift=4cm]
	
	\foreach \x [count=\p] in {0,...,7} {
		\node[shape=circle, inner sep=0.0mm, minimum size=0.0mm, opacity=0] (\p) at (-\x*60:1.5) {};
		\fill  (\p) circle (1.2pt);};

	\foreach \x [count=\p] in {0,...,7} {
		\node[shape=circle, inner sep=0.0mm, minimum size=0.0mm, opacity=0] (\p2) at (-\x*60:1) {};
		\fill  (\p2) circle (1.2pt);};
	\foreach \x [count=\p] in {0,...,7} {
		\node[shape=circle, inner sep=0.0mm, minimum size=0.0mm, opacity=0] (\p3) at (-\x*60+60:1) {};
		\fill  (\p3) circle (1.2pt);};
	\foreach \x [count=\p] in {0,...,7} {
		\node[color=black, shape=circle, inner sep=0mm, minimum size=0mm, opacity=0] (\p4) at (-\x*60-60:1) {};
		\fill  (\p4) circle (1.2pt);};

	\foreach \x [count=\p] in {1,...,7} {
		\draw[black!100!white, line width=1pt] (\p) -- (\p2);
		\draw[black!100!white, line width=1pt] (\p) -- (\p3);
		\draw[black!100!white, line width=1pt] (\p) -- (\p4);
		};
	\end{scope}
	\end{tikzpicture}
    \caption{$C_5^{\bowtie}$ (the M\"obius ladder) and $C_6^{\bowtie}$.}
    \label{fig:cycles}
\end{figure}

In~\cite{H10}, Hatami asked whether two particular graphs, the M\"obius strip and $C_{2k}\square K_2$, are weakly norming.
Theorem~\ref{thm:bowtie} hence answers both questions at once. 
We remark that every $C_{2k}^{\bowtie}$ is known to be Sidorenko by~\cite{KLL14}, but except the case~$C_3^{\bowtie}\cong K_{3,3}$ it is still an open question whether every~$C_{2k+1}^{\bowtie}$ is Sidorenko or not.

\medskip

Our proof relies on determining an equivalent condition of the (weakly) norming property. A function $f$ defined on the set of graphons is a \emph{(signed-)graphon parameter} if $f(W)=f(W')$ for (signed) graphons $W$ and $W'$ for which there exists a measure-preserving bijection $\varphi:[0,1]\rightarrow[0,1]$ satisfying $W(\varphi(x),\varphi(y))=W'(x,y)$. In particular, $t_H(W)$ is always a graphon parameter for any graph $H$.

\begin{theorem}\label{thm:main}
Let $H$ be a graph. Then
\begin{enumerate}[(i)]
    \item $H$ is weakly norming if and only if $t_H(\cdot)$ is a convex graphon parameter.
    \item $H$ is norming if and only if $t_H(\cdot)$ is a strictly convex signed graphon parameter.
\end{enumerate}
\end{theorem}
By using Theorem~\ref{thm:main}(ii), we also prove that $K_{t,t}$ minus a perfect matching, proven to be weakly norming by Lov\'asz, is not norming if $t>3$.
\begin{theorem}\label{thm:K55}
    For every $t>3$, $K_{t,t}$ minus a perfect matching is not norming. 
\end{theorem}
As observed by Hatami~\cite[Observation 2.5(ii)]{H10}, every norming graph must be \emph{eulerian}, i.e., every vertex has even degree. Thus, we only prove Theorem~\ref{thm:K55} for odd integers $t$, which gives the first examples of weakly norming graphs that are eulerian but not norming.

\section{Preliminaries}
Given an $n\times n$ symmetric matrix $A=(a_{ij})$, let $U_A$ be the two-variable symmetric step function on~$[0,1]^2$ defined by 
\begin{align*}
    U_{A}(x,y) = a_{ij}, \text{ if } (i-1)/n\leq x<i/n \text{ and } (j-1)/n\leq y<j/n
\end{align*}
and $U_A=0$ on the measure-zero set $x=1$ or $y=1$ for simplicity.
Trivially, $A\mapsto U_A$ is a linear map and $U_A$ satisfies the identity
\begin{align*}
    t_H(U_A) = n^{-v(H)} \sum_{\phi:V(H)\rightarrow [n]} \prod_{uv\in E(H)}a_{\phi(u)\phi(v)}.
\end{align*}
In other words, $t_H(U_A)$ is $n^{-v(H)}$ times a homogeneous $\binom{n+1}{2}$-variable polynomial of degree $e(H)$. We call the polynomial $P_{H,n}(A)$ for $A\in \Sy_n$, where $\Sy_n$ denotes the $\binom{n+1}{2}$-dimensional vector space of $n\times n$ real symmetric matrices. 

\medskip

The \emph{cut norm} $\norm{\cdot}_{\square}$ on $\W$ is defined by
\begin{align*}
    \norm{W}_\square := \sup_{S,T\subseteq [0,1]}\left\vert\int_{S\times T} W(x,y)dx dy\right\vert.
\end{align*}
Then the corresponding counting lemma is stated as follows:
\begin{lemma}[\cite{L12}, Exercise~10.28]\label{lem:counting}
Let $U$ and $W$ be signed graphons and let $H$ be a graph. Then
\begin{align*}
    |t_H(U) - t_H(W)|\leq 4e(H)\norm{U-W}_\square.
\end{align*}
\end{lemma}
The following lemma, which connects a (signed) graphon $W$ to a step function of the form $U_A$, is an easy consequence of the fact $\norm{W}_{\square}\leq \norm{W}_1$ and the dominated convergence theorem.
\begin{lemma}\label{lem:reg}
Let $W$ be a signed graphon. For every $\varepsilon>0$, there exists a symmetric matrix $A$ such that
$\norm{W-U_A}_\square <\varepsilon$.
\end{lemma}

To prove Theorem~\ref{thm:main}(ii), we shall use some facts about norming graphs, appeared in~\cite{L12}.
\begin{lemma}[\cite{L12},~Exercise~14.8]\label{lem:norming}
Let $H$ be a norming graph. Then $t_H(W)$ is always positive for a nonzero signed graphon $W$.
In particular, $e(H)$ is even, since $t_H(-W)=(-1)^{e(H)}t_H(W)$. 
\end{lemma}

\medskip

We follow the standard notion of convexity and related definitions. A \emph{convex set} is a subset~$C$ of a vector space such that $\lambda x+(1-\lambda)y\in C$ whenver $x,y\in C$ and $\lambda\in(0,1)$.
A function $f:C\rightarrow\mathbb{R}$ is said to be \emph{convex} if, for each $0<\lambda<1,$
$$f(\lambda x+(1-\lambda)y)\leq \lambda f(x)+(1-\lambda)f(y).$$
We say that a function $f$ is \emph{strictly convex} if the inequality above is strict whenever $x$ and $y$ are distinct.
We shall use a simple fact about convexity repeatedly in what follows:
\begin{lemma}\label{lem:composition}
Let $U$ be a convex subset of a vector space and let $f$ be a convex nonnegative function on $U$.
If $g:\mathbb{R}_{\geq 0}\rightarrow\mathbb{R}$ is an increasing convex function, then $g\circ f$ is also convex.
\end{lemma}
\begin{proof}
    Let $u,v\in U$. Then for each $\lambda\in (0,1)$.
    \begin{align*}
        g(f(\lambda u+(1-\lambda)v))\leq g(\lambda f(u)+(1-\lambda)f(v))
        \leq \lambda g(f(u))+(1-\lambda)g(f(v)),
    \end{align*}
    where the first inequality uses convexity of $f$ and monotonicity of $g$ and the second uses convexity of $g$.
\end{proof}

For a real-valued function $f(x_1,\cdots,x_n)$ that is twice differentiable on an open set $U\subseteq\mathbb{R}^n$,
the \emph{Hessian} of $f$, denoted by $\nabla^2 f$, is the $n\times n$ matrix $H=(h_{ij})$, where $h_{ij}=\frac{\partial^2 f}{\partial x_i\partial x_j}$. 
We will only consider polynomials $f$, so its Hessian $\nabla^2 f$ is always a symmetric matrix with polynomial-valued entries.
Standard results in convex analysis, e.g., Section 3.1.4 in~\cite{Convex04}, imply the following equivalence.
\begin{lemma}\label{lem:Hessian}
Every $n$-variable polynomial $P$ is convex on a convex set $C\subseteq\mathbb{R}^n$ if and only if its Hessian $\nabla^2 P$ is positive semidefinite on the interior of $C$.
\end{lemma}

\medskip

We also recall a basic fact in functional analysis.
\begin{lemma}\label{lem:ball}
Let $f$ be a nonnegative convex function on a vector space $V$ such that $f(x)=0$ if and only if $x=0$, and $f(\lambda x)=|\lambda|f(x)$. If $B:=\{x\in V:f(x)\leq 1\}$ is convex, then $f$ defines a norm on~$V$.
\end{lemma}
\begin{proof}
    It is enough to prove the triangle inequality $f(x+y)\leq f(x)+f(y)$.
    For nonzero $x$ and $y$, both $x_1:=x/f(x)$ and $y_1:=y/f(y)$ lie in the convex set $B$.
    Set $\lambda = \frac{f(x)}{f(x)+f(y)}$. Then by convexity, $\lambda x_1+(1-\lambda)y_1\in B$,
    and thus,
    \begin{align*}
        f\left(\frac{x+y}{f(x)+f(y)}\right)=f(\lambda x_1+(1-\lambda)y_1) \leq 1. 
    \end{align*}
    This proves subadditivity of $f$.
\end{proof}

\section{Convexity and weakly norming graphs}

Theorem~\ref{thm:main}(i) is a consequence of the following equivalence.
\begin{theorem}\label{thm:equiv}
Let $H$ be a graph. Then the following are equivalent:
\begin{enumerate}[(i)]
    \item $H$ is weakly norming.
    \item $t_H(\cdot)$ is a convex graphon parameter.
    \item $P_{H,n}(\cdot)$ is a convex polynomial on the positive orthant for every $n\in\mathbb{N}$.
\end{enumerate}
\end{theorem}
\begin{proof}
(i)$\Rightarrow$ (ii). If $\norm{\cdot}_{r(H)}$ is convex, then by Lemma~\ref{lem:composition}, $t_H(W)=\norm{W}_{r(H)}^{e(H)}$ is also convex on the set of graphons.

\medskip

\noindent
(ii)$\Rightarrow$(i). 
Convexity of $t_H(\cdot)$ for graphons naturally extends to all $U,W\in\W$ with nonnegative values. 
Thus, for all $U,W\in\W$ and $\lambda\in(0,1)$, 
\begin{align*}
    t_H(|\lambda W+(1-\lambda)U|)\leq t_H(\lambda |W|+(1-\lambda)|U|)
    \leq \lambda t_H(|W|)+(1-\lambda)t_H(|U|).
\end{align*}
Indeed, $0\leq W'\leq W$ pointwise implies $0\leq t_H(W')\leq t_H(W)$, which gives the first inequality, and the second follows from convexity of $t_H(\cdot)$.
Therefore, the set
\begin{align*}
    B:=\{W\in\W: t_H(|W|)\leq 1\}=\{W\in\W: t_H(|W|)^{1/e(H)}\leq 1\}
\end{align*}
is convex. Lemma~\ref{lem:ball} now proves the triangle inequality for $\norm{\cdot}_{r(H)}$.

\medskip

\noindent
(ii)$\Rightarrow$(iii). Let $A=(a_{ij})$ and $B=(b_{ij})$ be two $n\times n$ symmetric matrices with positive entries. We may assume that $\max a_{ij}\leq 1$ and $\max b_{ij}\leq 1$. Then convexity of $P_{H,n}$ immediately follows from linearity of the map $A\mapsto U_A$ and convexity of $t_H(\cdot)$ for graphons.

\medskip

\noindent
(iii)$\Rightarrow$(ii). Let $W_1$ and $W_2$ be two graphons.
By Lemma~\ref{lem:reg}, there exist $n\times n$ symmetric matrices $A_{1,n}$ and $A_{2,n}$ such that
$\norm{W_i -U_{A_{i,n}}}_{\square}\rightarrow 0$ as $n\rightarrow\infty$ for each $i=1,2$.
Convexity of $P_{H,n}$ gives
\begin{align*}
    t_H(\lambda U_{A_{1,n}}+(1-\lambda) U_{A_{2,n}})
    \leq \lambda t_H(  U_{A_{1,n}})+ (1-\lambda) t_H(U_{A_{2,n}}).
\end{align*}
Letting $n\rightarrow\infty$ finishes the proof, as $t_H(W_n)\rightarrow t_H(W)$ if $\norm{W_n-W}_{\square}\rightarrow 0$ by Lemma~\ref{lem:counting}.

\end{proof}
\begin{remark}
After proving the statement, we found that the equivalence between (i) and (ii) in fact implicitly appeared in Dole\v{z}al et al.~\cite{DGHRR} by a different approach using weak$^*$ limits. We include our shorter proof for the sake of completeness.
\end{remark}

\medskip

In particular, (iii) enables a computational way of verifying weakly norming property, by using Lemma~\ref{lem:Hessian}.
\begin{corollary}\label{cor:psd}
    A graph $H$ is weakly norming if and only if the Hessian $\nabla^2 P_{H,n}(A)$ is positive semidefinite for every $A\in \Sy_n$ with positive entries and $n\in\mathbb{N}$.
\end{corollary}

To prove Theorem~\ref{thm:bowtie}, we need some auxiliary facts about $C_k^{\bowtie}$.
For a vertex subset $X\subseteq V(H)$, let $N^*(X):=N(X)\setminus X$, where $N(X)$ denotes the \emph{union} of all neighbours of $x\in X$.
\begin{lemma}\label{lem:bowtie}
Let $H=C_k^{\bowtie}$ for $k>4$. Then
\begin{enumerate}[(i)]
    \item there is an edge $e$ in $H$ such that $N^*(e)$ induces exactly one edge, i.e., $e$ is contained in exactly one $4$-cycle, and
    \item if $X$ spans exactly two edges, then $N^*(X)$ contains an edge.
\end{enumerate}
\end{lemma}
Note that $C_3^{\bowtie}\cong K_{3,3}$ and $C_4^{\bowtie}\cong Q_3$ violate (i).
We omit the proof, as it is seen by a straightforward case analysis.

\begin{proof}[Proof of Theorem~\ref{thm:bowtie}]
    Let $H$ be the graph $C_{k}^{\bowtie}$.
    Since $\nabla^2 P_{H,n}(A)$ is a matrix with polynomial entries, its positive semidefiniteness for $A\in\Sy_n$ with positive entries extends to those $A\in\Sy_n$ with nonnegative entries. 
    We analyse a $2\times 2$ submatrix of the Hessian $\nabla^2P_{H,3}(A)$, where
    \begin{align*}
        A=\left[\begin{array}{ccc}
        1 & 1 & 0  \\
        1 & 0 & 1 \\
        0 & 1 & 0 
    \end{array}\right],
    \end{align*}
     with respect to the two variables $a_{13}$ and $a_{33}$. Namely, we write $h(x,y):=P_{H,3}(A_{x,y})$, where
     \begin{align*}
        A_{x,y}:=\left[\begin{array}{ccc}
        1 & 1 & y  \\
        1 & 0 & 1 \\
        y & 1 & x 
    \end{array}\right],
    \end{align*}
    and claim that $\nabla^2 h(x,y)$ is not positive semidefinite at $x=y=0$.
    
    We may decompose $h$ into $h(x,y) = q(x,y)+ \ell(x,y)+r(y)$, where $q(x,y)$ is the sum of all monomials with $x$-degree at least two, $\ell(x,y)$ is the sum of all monomials with $x$-degree one, and~$r(y)$ is the rest only depending on $y$. Then the Hessian $\nabla^2 h(0,0)$ is the matrix
    \begin{align*}
        \left[\begin{array}{ccc}
        q_{xx}(0,0) & \ell_{xy}(0,0) \\
        \ell_{xy}(0,0) & r_{yy}(0)
    \end{array}\right]
    \end{align*}
    and we claim that $q_{xx}(0,0)=0$ and that $\ell_{xy}(0,0)>0$.
    We regard $A_{x,y}$ as a looped, simple, and edge-weighted graph on $\{1,2,3\}$ with the weight $a_{ij}$ for each edge $ij$.
    Then $q(x,y)$ counts the weight on the homomorphisms from $H$ to $A_{x,y}$ that use the $x$-edge at least twice. 
    
    If a homomorphism uses the $x$-edge more than twice, then the corresponding monomial is divisible by $x^3$ and vanishes in $q_{xx}(0,0)$.
    Thus, to compute $q_{xx}(0,0)$, we only count those $H$-homomorphisms which use the $x$-edge exactly twice.
    Suppose that $e_1,e_2\in E(H)$ is mapped to the vertex $3$ with the looped $x$-edge. If a vertex in $N^*(e_1\cup e_2)$ is mapped to the vertex $1$, the homomorphism uses $y$-edge and the corresponding monomial vanishes in $q_{xx}(0,0)$. Otherwise if all the vertices in $N^*(e_1\cup e_2)$ are mapped to the vertex $2$, an edge contained in $N^*(e_1\cup e_2)$, which exists by Lemma~\ref{lem:bowtie}(ii), receives the loop weight 0. Thus, $q_{xx}(0,0)=0$.
    
    It remains to prove $\ell_{xy}(0,0)> 0$. By Lemma~\ref{lem:bowtie}(i), there is an edge  $e$ contained in at most one 4-cycle. Let $e'=uv$ be the edge disjoint from $e$ in the 4-cycle that contains  $e$. Consider the homomorphism that maps an edge $e$ to the $x$-edge, i.e., both ends of $e$ to $3$, exactly one end $u$ of $e'$ to $1$, all vertices in $N^*(e)\setminus\{ u\}$ to $2$, and the other vertices to $1$. Since $N^*(e)\setminus\{ u\}$ is an independent set by the uniqueness of the 4-cycle containing $e$, this is a homomorphism that uses both $x$- and $y$-edge exactly once. Thus, the corresponding monomial is $xy$, which proves that $\ell_{xy}(0,0)\geq 1$.
\end{proof}

\section{Strict convexity and norming graphs}

Theorem~\ref{thm:main}(ii) follows from a result analogous to Theorem~\ref{thm:equiv}.
\begin{theorem}\label{thm:eqv_norming}
Let $H$ be a graph. Then the following are equivalent:
\begin{enumerate}[(i)]
    \item $H$ is norming.
    \item $t_H(\cdot)$ is a strictly convex parameter for signed graphons.
    \item $P_{H,n}(\cdot)$ is a strictly convex polynomial on $\Sy_n$ for every $n\in\mathbb{N}$.
    \item $P_{H,n}(\cdot)^{1/e(H)}$ is a norm on $\Sy_n$.
\end{enumerate}
\end{theorem}
\begin{proof}
    \noindent
    (i)$\Rightarrow$(ii). Let $U,W\in\mathcal{W}$. Hatami proved the following inequality (see (34) in \cite{H10}):
    \begin{align*}
        t_H(U+W)+t_H(U-W) \leq 2^{e(H)-1}(t_H(U)+t_H(W)).
    \end{align*}
    Since $H$ is norming, $t_H(U-W)>0$ unless $U=W$ almost everywhere by Lemma~\ref{lem:norming}. This implies strict convexity of $t_H(\cdot)$.
    
    \medskip
    
    \noindent (ii)$\Rightarrow$(iii). This immediately follows from the linearity of the map $A\mapsto U_A$.
    
    \medskip
    
    \noindent (iii)$\Rightarrow$(iv). 
    If $e(H)$ is odd, then $P_{H,n}(A)+P_{H,n}(-A)=0$ for every $A\in\Sy_n$, which contradicts strict convexity. Thus, $e(H)$ is even. Again by strict convexity, $2P_{H,n}(A)=P_{H,n}(A)+P_{H,n}(-A)>0$ whenever $A\neq 0$. Hence, $P_{H,n}(A)^{1/e(H)}$ is well-defined and positive for every nonzero $A$. Furthermore, $P_{H,n}(\lambda A)^{1/e(H)}=|\lambda| P_{H,n}(A)^{1/e(H)}$. 
     Since  
     \begin{align*}
         B:=\{A\in\Sy_n: P_{H,n}(A)\leq 1\}=\{A\in\Sy_n: P_{H,n}(A)^{1/e(H)}\leq 1\}
     \end{align*}
     is a convex set, we may apply Lemma~\ref{lem:ball} and conclude that $P_{H,n}(A)$ is a norm on $\Sy_n$.
     
    \medskip
    \noindent (iv)$\Rightarrow$(i). The proof is the same as the part (iii)$\Rightarrow$(ii) of Theorem~\ref{thm:equiv}.
\end{proof}
Indeed, positive definiteness of the Hessian implies strict convexity of a polynomial, but the converse is not true in general. 
Thus, the naive analogue of Corollary~\ref{cor:psd} obtained by replacing weakly norming and positive semidefinite by norming and positive definite, respectively, does not hold. One might still hope to prove that a graph $H$ is norming by showing that the Hessian $\nabla^2 P_{H,n}(A)$ is positive definite at each nonzero $A\in\Sy_n$, using the one-sided implication. However, we show that this is impossible by proving that every norming graph has a singular Hessian $\nabla^2 P_{H,n}(A)$ at some $A\neq 0$ whenever $n$ is even.
\begin{proposition}\label{prop:pdhessian}
    For every $n$, There exists a nonzero $2n\times 2n $ symmetric matrix $A$ such that $\nabla^2 P_{H,2n}(A)$ is singular for every norming graph $H$.
\end{proposition}
\begin{proof}
    Let $A=\left[\begin{array}{cc}
        J_n & -J_n\\ -J_n & J_n 
    \end{array}\right]$, where $J_n$ denotes the $n\times n$ matrix with all entries equal to 1. We claim that $\nabla^2P_{H,2n}(A)$ has eigenvalue 0 with the eigenvector $1_n=(1,1,\cdots,1)^T\in \mathbb{R}^{n(2n+1)}$. 
    Recall the folklore fact~\cite[Example~5.14]{L12} that $t_F(U_A)$ is the indicator function that $F$ is eulerian. In particular, $H$ is eulerian and $e(H)$ is even.
    Thus,
    \begin{align*}
        t_H(U_A-\varepsilon)+t_H(U_A+\varepsilon) = 2t_H(U_A) + 2\sum_{J}t_J(U_A)\varepsilon^{e(H)-e(J)}
        = 2 + 2\varepsilon^2\sum_{F}t_F(U_A) +O(\varepsilon^4),
    \end{align*}
    where the first sum is taken over all proper subgraphs $J$ of $H$ with even number of edges and the second is taken over all subgraphs $F\subseteq H$ with $e(F)=e(H)-2$. Since one always obtains a non-eulerian subgraph $F$ by deleting two edges from an eulerian graph $H$, $t_{F}(U_A)=0$. Thus, $t_H(U_A-\varepsilon)+t_H(U_A+\varepsilon)=2+O(\varepsilon^4)$.
    On the other hand, by the Taylor expansions of $P_{H,2n}$ at $A$,
    \begin{align*}
        P_{H,2n}(A+\varepsilon J_{2n})+P_{H,2n}(A-\varepsilon J_{2n}) = P_{H,2n}(A)+ 2\varepsilon^2 1_n^T\nabla^2 P_{H,2n}(A)1_n +O(\varepsilon^3).
    \end{align*}
    Since $P_{H,2n}(A+\varepsilon J_{2n})+P_{H,2n}(A-\varepsilon J_{2n})=(2n)^{v(H)}(t_H(U_A-\varepsilon)+t_H(U_A+\varepsilon))$, it follows that $1_n^T\nabla^2 P_{H,2n}(A)1_n=0$. Since $\nabla^2 P_{H,2n}(A)$ is positive semidefinite, $\nabla^2 P_{H,2n}(A)1_n$ must be zero. This completes the proof of the claim.
\end{proof}

As already used in the last line of the proof, we are only able to obtain a weaker analogue of Corollary~\ref{cor:psd}. 
\begin{corollary}\label{cor:psd_norming}
    For a norming graph $H$, $\nabla^2 P_{H,n}(A)$ is positive semidefinite for every~$A\in \Sy_n$.
\end{corollary}

It is still enough to find $A\in \Sy_n$ such that $\nabla^2 P_{H,n}(A)$ is not positive semidefinite to prove that~$H$ is not norming. This is exactly what we do in the proof of Theorem~\ref{thm:K55}.

\begin{proof}[Proof of Theorem~\ref{thm:K55}]
     Let $H_t$ be the graph $K_{2t+1,2t+1}\setminus (2t+1)\cdot K_2$. As mentioned before, it is enough to prove that $H_t$ is not norming, as $K_{2t,2t}$ minus a perfect matching is not eulerian and thus not norming.
     Let
    \begin{align*}
        A=\left[\begin{array}{ccc}
        x & y & \varepsilon  \\
        y & 1 & 1 \\
        \varepsilon & 1 & -1 
    \end{array}\right]
    \end{align*}
    and let $h(x,y):=P_{H,3}(A)$. Here we suppress the dependency on $0<\varepsilon<1$, since $\varepsilon$ is a small constant to be chosen later.
    We analyse the $2\times 2$ Hessian matrix $\nabla^2 h$ at $(0,0)$.
    As in the proof of Theorem~\ref{thm:bowtie},
    we decompose $h(x,y)$ into three parts, i.e., 
    \begin{align*}
        h(x,y) = q(x,y) + \ell(x,y) + r(y), 
    \end{align*}
    where $q(x,y)$ is the sum of monomials divisible by $x^2$, $\ell$ is the sum of monomials whose $x$-degree is~$1$, and $r$ is the remaining terms. 
    Then the Hessian $\nabla^2 h(0,0)$ is the matrix
    \begin{align*}
        \left[\begin{array}{ccc}
        q_{xx}(0,0) & \ell_{xy}(0,0) \\
        \ell_{xy}(0,0) & r_{yy}(0)
    \end{array}\right].
    \end{align*}
    We regard $A$ as a looped, simple, and edge-weighted graph on $\{1,2,3\}$ with the weight $a_{ij}$ for each edge $ij$. For the same reason as in the proof of Theorem~\ref{thm:bowtie}, $q_{xx}(0,0)$ is equal to the number of homomorphisms that use the $x$-edge exactly twice without using the $y$-edge. 
    Such a homomorphism~$\phi$ maps at least three vertices $V_1$ in $H_t$ that induce exactly two edges to the vertex $1$ and never maps their neighbour to the vertex~$2$. Thus, $N^*(V_1)$ must be embedded to the vertex $3$. Since  $H_t$ is $2t$-regular and $V_1$ contains exactly two edges, $e(V_1,N^*(V_1))\geq 6t-4$ and thus, $\phi$ uses the $\varepsilon$-edge at least $6t-4$ times. 
    
    Analogously, $\ell_{xy}(0,0)$ counts the number of homomorphisms that use both the $x$- and $y$-edges exactly once and hence, use the $\varepsilon$-edge at least $4t-3$ times. The homomorphisms using the $y$-edge exactly twice and avoiding the $x$-edge must use $\varepsilon$-edge at least $2t-2$ times. Therefore,
    \begin{align*}
       \nabla^2 h(0,0)=
       \left[\begin{array}{ccc}
        q_{xx}(0,0) & \ell_{xy}(0,0) \\
        \ell_{xy}(0,0) & r_{yy}(0)
    \end{array}\right]=
       \left[\begin{array}{ccc}
        O(\varepsilon^{6t-4}) & O(\varepsilon^{4t-3}) \\
        O(\varepsilon^{4t-3}) & O(\varepsilon^{2t-2})
    \end{array}\right].
    \end{align*}
    Here $O(\cdot)$ notation includes implicit multiplicative constants depending only on $t$.
    
    Unfortunately, the product of the diagonal entries and the product of the off-diagonal entries are in the same order $O(\varepsilon^{8t-6})$.
    However, we claim that $r_{yy}(0)$ is asymptotically smaller than $O(\varepsilon^{2t-2})$ and also that $|\ell_{xy}(0,0)|=\Omega(\varepsilon^{4t-3})$, which implies that $\nabla^2 h(0,0)$ is not positive semidefinite for a sufficiently small $\varepsilon>0$.
    
    Let~$A\cup B$ be the bipartition of~$H_t$ and let~$A=\{a_1,\cdots,a_{2t+1}\}$ and~$B=\{b_1,\cdots,b_{2t+1}\}$ such that~$a_ib_i$, $1\leq i\leq 2t+1$, is the missing perfect matching in~$H_t$.
    Firstly, let $\Phi_{yy}$ be the set of homomorphisms that use the~$y$-edge twice and the~$\varepsilon$-edge exactly~$2t-2$ times while avoiding the~$x$-edge.
    Each~$\varphi\in \Phi_{yy}$ must map one vertex, say~$a_1$, to~$1$, two neighbours of~$a_1$ to~$2$, and the other~$2t-2$ neighbours of~$a_1$ to~$3$.
    That is, once we choose the vertex~$a_1$ and two of its neighbours to be mapped to~$2$, all the embeddings of the neighbours of~$a_1$ are fixed. Consider these vertices as pre-embedded.
    Let~$V_3$ be the set of~$2t-2$ vertices mapped to~$3$ and let~$U$ be the vertices that are not yet embedded. Then~$U=\{a_2,\cdots,a_{2t+1}\}\cup\{b_1\}$.
    In particular, $U$ induces a star centred at~$b_1$ with~$2t$ edges. Also note that by the definition of~$\Phi_{yy}$, the homomorphisms in~$\Phi_{yy}$ do not map any other vertex than~$a_1$ to~$1$.
    For each~$\varphi\in\Phi_{yy}$, denote by $U_{\varphi}$ the subset of~$U$ mapped to the vertex~$3$. Then the coefficient of the term $\varepsilon^{2t-2}y^{2}$ in~$r(y)$ is determined by
    \begin{align}\label{eq:phiweight}
        \sum_{\varphi\in\Phi_{yy}} (-1)^{e(V_3,U_{\varphi})+e(V_3)+e(U_{\varphi})}.
    \end{align}
    Suppose $b_1\in U_{\varphi}$. For each $\varphi$, let $b_i$ and $b_j$, $i<j$, be the two vertices mapped to the vertex $2$. Then both $a_i$ and $a_j$ have all their $2t-1$ other neighbours than $b_i$ and $b_j$ mapped to $3$. Thus, by switching the image of $a_i$ under $\varphi$ between $2$ and $3$, we produce another homomorphism $\overbar{\varphi}$ whose weight $(-1)^{e(V_3,U_{\overbar{\varphi}})+e(V_3)+e(U_{\overbar{\varphi}})}$ has exactly the opposite sign of that of $\varphi$. This switching is an involution, and thus, the two terms pair up and cancel each other in~\eqref{eq:phiweight}.
    If $b_1\notin U_{\varphi}$, then one may do an analogous switching with the minimum indexed vertex amongst $a_2,\cdots,a_{2t+1}$ that has an odd degree to those vertices mapped to $3$. Thus, \eqref{eq:phiweight} evaluates to zero.
    
    To prove $|\ell_{xy}(0,0)|=\Omega(\varepsilon^{4t-3})$, let $\Psi_{xy}$ be the set of homomorphisms that use each of the $x$- and $y$-edge exactly once. Suppose that, under $\psi\in\Psi_{xy}$, $a_i$ and $b_{j}$, $i,j>1$ and $i\neq j$, are mapped to $1$ and $b_{k}$, $i\neq k>1$, is  mapped to $2$. To avoid using the $y$-edge more than once, $\psi$ must map $(B\setminus\{b_i,b_j,b_k\})\cup (A\setminus\{a_i,a_j\})$ to the vertex $3$. Thus, there are only two vertices $a_j$ and $b_i$ whose embedding is not yet determined.
    Note that $a_j$ and $b_i$ have $2t-2$ and $2t-1$ neighbours mapped to $3$, respectively, and they are adjacent. Let $\alpha_{\psi}$ and $\beta_\psi$ be the indicator function that $a_j$ and $b_i$ are mapped to $3$ by $\psi$, respectively. Then the coefficient of the term $\varepsilon^{4t-3}xy$ in $\ell(x,y)$ is a nonzero constant times
    \begin{align*}
        \sum_{\psi\in\Psi_{xy}}(-1)^{(2t-2)\alpha_\psi+(2t-1)\beta_\psi+\alpha_\psi\beta_\psi}=\sum_{\psi\in\Psi_{xy}}(-1)^{\beta_\psi+\alpha_\psi\beta_\psi}.
    \end{align*}
    Since each choice $(\alpha_\psi,\beta_\psi)\in\{0,1\}^2$ determines a homomorphism $\psi\in\Psi_{xy}$, $(\alpha_\psi,\beta_\psi)$ is uniformly distributed on $\{0,1\}^2$. Hence, the sum above evaluates to a nonzero constant, which proves the claim.
\end{proof}

\section{Concluding remarks}
Our method using the Hessian matrix $\nabla^2 P_{H,n}$ is reminiscent of~\cite{KMPW19} in the sense that both rely on determining positive semidefiniteness of matrices given by homomorphism counts. More precisely, in~\cite{KMPW19} they looked at two edges $e$ and $e'$ in a graph $G$ sharing a vertex and used non-positive semidefiniteness of the $2\times 2$ matrix 
\begin{align*}
        A_{e,e'}=\left[\begin{array}{ccc}
        h_{e,e} & h_{e,e'} \\
        h_{e,e'} & h_{e',e'}
    \end{array}\right],
\end{align*}
where $h_{e_1,e_2}$ counts the number of those homomorphisms from $H$ to $G$ which map a $K_{1,2}$ in $H$ to the homomorphic copy of $K_{1,2}$'s formed by $e_1$ and $e_2$,
to prove that a certain $H$ is not weakly norming. 

This is somewhat analogous to the Hessian matrix obtained by evaluating the corresponding weights of $e$ and $e'$ to be zero. However, the Hessian does not take the particular $K_{1,2}$-structure into account, so it has larger entries than $A_{e,e'}$ above. 
We did not attempt to reprove their result using our language, but we remark that there are non-weakly norming graphs that satisfy their positive semidefiniteness condition. For instance, take a vertex-disjoint union of two non-isomorphic connected weakly norming graphs. This is proven to be not weakly norming in~\cite{GHL19}, but the corresponding $2\times 2$ matrix in~\cite{KMPW19} is positive semidefinite, since it is a positive linear combination of the respective matrices of the components. It would be interesting to see if the two distinct positive semidefiniteness conditions are equivalent for connected graphs $H$.

\vspace{5mm}

\noindent
\textbf{Acknowledgements.} We would like to thank Sasha Sidorenko for suggesting us to prove Theorem~\ref{thm:K55} and Jan Hladk\'{y} for explaining his result~\cite{DGHRR} with colleagues obtaining part of Theorem~\ref{thm:main}. We are also grateful to David Conlon, Christian Reiher, and Mathias Schacht for helpful comments and discussions.

\bibliographystyle{abbrv}
\bibliography{references}

\end{document}